\documentclass{amsart}

\usepackage{graphicx,bbm,graphics,hyperref}
\usepackage[utf8]{inputenc}

\newtheorem{theorem}{Theorem}[section]
\newtheorem{lemma}[theorem]{Lemma}
\theoremstyle{definition}
\newtheorem{definition}[theorem]{Definition}
\newtheorem{example}[theorem]{Example}

\numberwithin{equation}{section}

\newlength\cellsize 

\savebox2{%
\begin{picture}(12,12)
\put(0,0){\line(1,0){12}}
\put(0,0){\line(0,1){12}}
\put(12,0){\line(0,1){12}}
\put(0,12){\line(1,0){12}}
\end{picture}}

\newcommand\cellify[1]{\def\thearg{#1}\def\nothing{}%
\ifx\thearg\nothing
\vrule width0pt height\cellsize depth0pt\else
\hbox to 0pt{\usebox2\hss}\fi%
\vbox to 12\unitlength{
\vss
\hbox to 12\unitlength{\hss$#1$\hss}
\vss}}

\newcommand\tableau[1]{\vtop{\let\\=\cr
\setlength\baselineskip{-12000pt}
\setlength\lineskiplimit{12000pt}
\setlength\lineskip{0pt}
\halign{&\cellify{##}\cr#1\crcr}}}

\begin{document}

\title{Distribution of descents in matchings}
\author{Gene B. Kim}
\address{Dept. of Mathematics, University of Southern California, Los Angeles, CA 90089}
\email{genebkim@usc.edu}


\date{October 9, 2017}


\begin{abstract}
The distribution of descents in a fixed conjugacy class of $S_n$ is studied, and it is shown that its moments have an interesting property. A particular conjugacy class that is of interest is the class of matchings (also known as fixed point free involutions). This paper provides a bijective proof of the symmetry of the descents and major indices of matchings and uses a generating function approach to prove an asymptotic normality theorem for the number of descents in matchings.
\end{abstract}

\maketitle

\section{Introduction}

The theory of descents in permutations have been studied thoroughly and is related to many questions. In \cite{Knuth}, Knuth connects descents with the theory of sorting and the theory of runs in permutations, and in \cite{Diaconis1}, Diaconis, McGrath, and Pitman study a model of card shuffling in which descents play a central role. Bayer and Diaconis also use descents and rising sequences to give a simple expression for the chance of any arrangement after any number of shuffles and uses this to give sharp bounds on the approach to randomness in \cite{Bayer}. Garsia and Gessel find a generating function for the joint distribution of descents, major index, and inversions in \cite{Garsia}, and Gessel and Reutenauer show that the number of permutations with given cycle structure and descent set is equal to the scalar product of two special characters of the symmetric group in \cite{Gessel}. Petersen also has an excellent and very thorough book on the subject of Eulerian numbers \cite{Petersen}.

\begin{definition}
A permutation $\pi \in S_n$ has a \textit{descent} at position $i$ if $\pi(i) > \pi(i+1)$, where $i = 1, \dots, n-1$, and the \textit{descent set} of $\pi$, denoted $Des(\pi)$ is the set of all descents of $\pi$. The \textit{descent number} of $\pi$ is defined as $d(\pi) := \lvert Des(\pi) \rvert + 1$.
\end{definition}

It is well known (\cite{Diaconis2}) that the distribution of $d(\pi)$ in $S_n$ is asymptotically normal with mean $\frac{n+1}{2}$ and variance $\frac{n-1}{12}$. Fulman also uses Stein's method to show that the number of descents of a random permutation satisfies a central limit theorem with error rate $n^{-1/2}$ in \cite{Fulman2}. In \cite{Chatterjee}, Chatterjee and Diaconis prove a central limit theorem for $d(\pi) + d(\pi^{-1})$, where $\pi$ is a random permutation.

One open question that arises from this is to look at the distribution of $d(\pi)$ for particular conjugacy classes of $S_n$. Using generating functions, Fulman proves the following analogous result in \cite{Fulman1} about conjugacy classes with large cycles only:

\begin{theorem}
For every $n \geq 1$, pick a conjugacy class $C_n$ in $S_n$, and let $n_i(C_n)$ be the number of $i$-cycles in $C_n$. Suppose that for all $i$, $n_i(C_n) \to 0$ as $n \to \infty$. Then, the distribution of $d(\pi)$ in $C_n$ is asymptotically normal with mean $\frac{n-1}{2}$ and variance $\frac{n-1}{12}$.
\end{theorem}

Note that the asymptotic mean and variance in the theorem is same as those of $S_n$.

\begin{definition}
A \textit{matching}, also known as a \textit{fixed point free involution}, is a permutation $\pi \in S_{2n}$ such that $\pi^2 = 1$ and $\pi(i) \neq i$ for all $i = 1,\dots,2n$.
\end{definition}

In this paper, we will consider the particular conjugacy class of matchings, which come up frequently. Goldstein and Rinott consider a permutation method for testing whether observations given in their natural pairing exhibit an unusual level of similarity in situtations where all pairings can be compared in \cite{Goldstein}. In \cite{Harper}, Harper shows that if the generating function of a finite sequence has real roots, then a central limit theorem follows. However, it turns out that the Eulerian function for matchings has only complex roots.

It is worth noting that the number of matchings in $S_{2n}$, denoted by $F_{2n}$ in this paper, is $F_{2n} = (2n-1)!! = \prod_{i=1}^n (2i-1)$. The main goal of this paper is to show that the distribution of $d(\pi)$ in matchings is asymptotically normal.

\section{Mean and variance of $d(\pi)$ and $maj(\pi)$ of matchings}

In this section, we will compute the mean and the variance of $d(\pi)$ and $maj(\pi)$ of matchings by the method of indicators. Given $\pi \in S_{2n}$, let $X_i$ be the event that $\pi$ has a descent at position $i$ and $\mathbbm{1}_i$ denote the indicator function of $X_i$.

\begin{lemma}\label{lem2.1}
If $\pi \in S_{2n}$ is a matching chosen uniformly at random, the probability that $\pi$ has a descent at position~$i$ is
\[
	P(X_i) = \frac{n}{2n-1}.
\]
\end{lemma}

\begin{proof}
Let $1 \leq i < 2n$ be arbitrary. We proceed by counting the number of matchings that have a descent at position~$i$. There are two cases to consider.

The first case is where $\pi(i) = i+1$. This automatically guarantees a descent at position~$i$, and there are $F_{2n-2}$ ways to match the remaining $2n-2$ elements. Hence, there are $F_{2n-2}$~matchings with a descent at position~$i$ and $\pi(i) = i+1$.

The second case is where $\pi(i) \neq i+1$. There are $\binom{2n-2}{2}$ ways of choosing pairs of candidates for $\left( \pi(i), \pi(i+1) \right)$. In order for there to be a descent at position $i$, the larger of the two candidates has to be matched with~$i$ and the smaller with~$i+1$. There are $F_{2n-4}$ ways to match the remaining $2n-4$ elements. Hence, there are $\binom{2n-2}{2}F_{2n-4}$ matchings with a descent at position~$i$ and $\pi(i) \neq i+1$.

Summing the two cases, we see that the probability that $\pi$ has a descent at position~$i$ is
\begin{equation*}
	P(X_i) = \frac{F_{2n-2} + \binom{2n-2}{2}F_{2n-4}}{F_{2n}} = \frac{n}{2n-1}.
\end{equation*}
\end{proof}

Note that $P(X_i)$ is independent of $i$. We can use this to calculate the mean of $d(\pi)$.

\begin{theorem}
Let $\pi \in S_{2n}$ be a matching chosen uniformly at random. Then, $\mathbb{E}d(\pi) = n + 1$.
\end{theorem}

\begin{proof}
The result follows by linearity of expectation and Lemma \ref{lem2.1}:
\begin{equation*}
	\mathbb{E}d(\pi) = \mathbb{E} \left( \sum_{i=1}^{2n-1} \mathbbm{1}_i + 1 \right) = (2n-1)\frac{n}{2n-1} + 1 = n + 1.
\end{equation*}
\end{proof}

Finding the variance is a bit more complicated. In order to find the second moment of $d(\pi)$, we proceed along in a similar fashion as when we found the mean.

\begin{lemma}\label{lem2.2}
If $n \geq 3$ and $\pi \in S_{2n}$ is a matching chosen uniformly at random, for $\lvert i - j \rvert = 1$, the probability that $\pi$ has descents at positions~$i$ and~$j$ is
\begin{equation*}
	P(X_i \cap X_j) = \frac{n + 1}{3(2n - 1)}.
\end{equation*}
\end{lemma}

\begin{proof}
Let $1 \leq i < i+1 = j < 2n$ be arbitrary. We proceed by counting the number of matchings that have descents at positions~$i$ and~$j$. There are two cases to consider.

The first case is where there is an intrapair within $i$, $i+1$, and $i+2$, i.e. $\lvert \left\{ i, i+1, i+2 \right\} \cap \left\{ \pi(i), \pi(i+1), \pi(i+2) \right\} \rvert = 2$. If two of the three get matched, there are $2n-3$ choices to be matched with the remaining of the three. If the choice, say $m$, is greater than $i$, then $i$ must be matched with $m$ and $i+1$ with $i+2$ for there to be descents at positions $i$ and $i+1$. If $m$ is less than $i$, $i+2$ must be matched with $m$ and $i$ with $i+1$ for there to be descents at positions $i$ and $i+1$. In both cases, there are $F_{2n-4}$ ways to match the remaining $2n-4$ elements, and so, there are $(2n-3)F_{2n-4}$ matchings in the first case.

The second case is where there is no intrapair within $i$, $i+1$, and $i+2$, i.e. $\lvert \left\{ i, i+1, i+2 \right\} \cap \left\{ \pi(i), \pi(i+1), \pi(i+2) \right\} = 0$. There are $\binom{2n-3}{3}$ ways to choose three elements to get matched with $i$, $i+1$, and $i+2$. The descent conditions necessitate the largest of the three choices be matched with $i$, the second largest with $i+1$, and the smallest with $i+2$. Since there are $F_{2n-6}$ ways to match the remaining $2n-6$ elements, there are $\binom{2n-3}{3} F_{2n-6}$ matchings in the second case.

Hence, if $\lvert i - j \rvert = 1$,
\begin{equation*}
	P(X_i \cap X_j) = \frac{(2n-3)F_{2n-4} + \binom{2n-3}{3} F_{2n-6}}{F_{2n}} = \frac{n + 1}{3(2n-1)}.
\end{equation*}
\end{proof}

\begin{lemma}\label{lem2.3}
If $n \geq 4$ and $\pi \in S_{2n}$ is a matching chosen uniformly at random, for $\lvert i - j \rvert > 1$, the probability that $\pi$ has descents at positions~$i$ and~$j$ is
\begin{equation*}
	P(X_i \cap X_j) = \frac{n(n-1)}{(2n-1)(2n-3)}.
\end{equation*}
\end{lemma}

\begin{proof}
Let $1 \leq i < j < 2n$ be arbitrary. In order to count the number of matchings that have descents at positions~$i$ and $j$, we consider the three following cases.

The first case is where there are two intrapairs within $i$, $i+1$, $j$, and $j+1$. There are two ways to have two intrapairs so that there are descents at positions~$i$ and~$j$: either $\pi(i) = i+1$ and $\pi(j) = j+1$ or $\pi(i) = j+1$ and $\pi(j) = i+1$. Since there are $F_{2n-4}$ ways to match up the $2n-4$ remaining elements, there are $2F_{2n-4}$ matchings in this case.

The second case is where there is one intrapair within $i$, $i+1$, $j$, and $j+1$. We break this down into subcases.
\begin{enumerate}
\item The first subcase is when the intrapair is within $i$ and $i+1$, i.e. $\pi(i) = i+1$. Then, there are $\binom{2n-4}{2}$ ways to choose two elements to get matched with $j$ and $j+1$. The larger element has to be matched with $j$ for there to be a descent at position~$j$, and since there are $F_{2n-6}$ ways to match the remaining $2n-6$ elements, there are $\binom{2n-4}{2} F_{2n-6}$ matchings in this subcase.
\item The second subcase is when the intrapair is within $j$ and $j+1$. This is the same as the previous subcase, and so, there are $\binom{2n-4}{2} F_{2n-6}$ matchings in this subcase.
\item The third subcase is when the intrapair is between one of $i$ and $i+1$ and one of $j$ and $j+1$. For this subcase, we first choose the candidates to be matched with the two elements that are not in the intrapair, and there are $\binom{2n-4}{1 \quad 1}$ ways to do this. The size of the candidates relative to $i$ and $j$ determine which of $i$ and $i+1$ get matched with which one of $j$ and $j+1$. For example, if $a$ and $b$ are the candidates chosen to be matched with one of $i$ and $i+1$ and one of $j$ and $j+1$, respectively, with $a < i$ and $b > j$, then $a$ would be matched with $i+1$, $b$ would be matched with $j$, and $i$ would be matched with $j+1$. For each pair $(a,b)$ of candidates chosen, one can see that there is a unique matching that guarantees descents at position $i$ and $j$. There are $F_{2n-6}$ ways to match the remaining $2n-6$ elements, and so, there are $\binom{2n-4}{1 \quad 1} F_{2n-6}$ matchings in this subcase.
\end{enumerate}
Summing the subcases, we see that there are $4\binom{2n-4}{2} F_{2n-6}$ matchings in this case.

The third case is where there is no intrapair within $i$, $i+1$, $j$, and $j+1$. There are $\binom{2n-4}{2 \quad 2}$ to choose two possible candidates to be matched with $i$ and $i+1$ and two possible candidates to be matched with $j$ and $j+1$. There are $F_{2n-8}$ ways to match the remaining $2n-8$ elements, and so, there are $6\binom{2n-4}{4} F_{2n-8}$ matchings in the third case.

Hence, if $\lvert i - j \rvert > 1$,
\begin{equation*}
	P(X_i \cap X_j) = \frac{2F_{2n-4} + 4\binom{2n-4}{2}F_{2n-6} + 6\binom{2n-4}{4}F_{2n-8}}{F_{2n}} = \frac{n(n-1)}{(2n-1)(2n-3)}.
\end{equation*}
\end{proof}

\begin{theorem}
Let $\pi \in S_{2n}$ be a matching chosen uniformly at random, where $n \geq 4$. Then,
\begin{equation*}
	\text{Var}\left( d(\pi) \right) = \frac{(n+4)(n-1)}{3(2n-1)}.
\end{equation*}
\end{theorem}

\begin{proof}
By Lemmas \ref{lem2.1}, \ref{lem2.2}, and \ref{lem2.3}, the second moment of $d(\pi)$ is
\begin{align*}
	\mathbb{E}\left( d(\pi) \right)^2 &= \left( \sum_{i=1}^{2n-1} \mathbb{E} \mathbbm{1}_i^2 + \sum_{\lvert i-j \rvert = 1} \mathbb{E}\mathbbm{1}_{i,j} + \sum_{\lvert i-j \rvert > 1} \mathbb{E}\mathbbm{1}_{i,j} \right) + 2\sum_{i=1}^{2n-1} \mathbb{E} \mathbbm{1}_i + 1 \\
	&= \frac{6n^3 + 10n^2 + 3n - 7}{3(2n-1)},
\end{align*}
and so, the variance of $d(\pi)$ is
\begin{equation*}
	\text{Var}\left( d(\pi) \right) = \mathbb{E}\left( d(\pi) \right)^2 - \left( \mathbb{E} d(\pi) \right)^2 = \frac{(n+4)(n-1)}{3(2n-1)}.
\end{equation*}
\end{proof}

Note that the asymptotic mean and variance of $d(\pi)$ of matchings are $n$ and $\frac{n+3}{6}$, respectively, which are different from those of a random permutation.

Another permutation statistic used a lot is the major index.

\begin{definition}
Given a permutation $\pi \in S_n$, the \textit{major index} of $\pi$ is the sum of the descent positions of $\pi$, i.e.
\begin{equation*}
	maj(\pi) = \sum_{i \in Des(\pi)} i.
\end{equation*}
\end{definition}

We can compute the mean and variance of $maj(\pi)$ in the same way as above.

\begin{theorem}
Let $\pi \in S_{2n}$ be a matching chosen uniformly and random. Then,
\begin{equation*}
	\mathbb{E}maj(\pi) = n^2.
\end{equation*}
\end{theorem}

\begin{proof}
By Lemma \ref{lem2.1}, we see that
\begin{equation*}
	\mathbb{E}maj(\pi) = \mathbb{E}\left( \sum_{i=1}^{2n-1} i\mathbbm{1}_i \right) = \frac{n}{2n-1} \sum_{i=1}^{2n-1} i = n^2.
\end{equation*}
\end{proof}

\begin{theorem}
Let $\pi \in S_{2n}$ be a matching chosen uniformly at random, where $n \geq 4$. Then,
\begin{equation*}
	\text{Var}\left( maj(\pi) \right) = \frac{2n(n+4)(n-1)}{9}.
\end{equation*}
\end{theorem}

\begin{proof}
By Lemmas \ref{lem2.1}, \ref{lem2.2}, and \ref{lem2.3}, the second moment of $maj(\pi)$ is
\begin{align*}
	\mathbb{E}\left( maj(\pi) \right)^2 &= \sum_{i=1}^{2n-1} i^2 \mathbb{E}\mathbbm{1}_i^2 + \sum_{\lvert i-j \rvert = 1} ij\mathbb{E}_{i,j} + \sum_{\lvert i-j \rvert >1} ij\mathbb{E}_{i,j} \\
	&= \frac{n}{2n-1}\sum_{i=1}^{2n-1} i^2 + \frac{n+1}{3(2n-1)} \sum_{\lvert i-j \rvert=1} ij + \frac{n(n-1)}{(2n-1)(2n-3)} \sum_{\lvert i-j \rvert>1} ij \\
	&= \frac{9n^4 + 2n^3 + 6n^2 - 8n}{9},
\end{align*}
and so, the variance of $maj(\pi)$ is
\begin{equation*}
	\text{Var}\left( maj(\pi) \right) = \mathbb{E}\left( maj(\pi) \right)^2 - \left( \mathbb{E} maj(\pi) \right)^2 = \frac{2n(n+4)(n-1)}{9}.
\end{equation*}
\end{proof}

\section{Symmetry of $d(\pi)$ and $maj(\pi)$}

While looking at the distribution of the descent numbers and major indices for $S_n$, where $n$ is even, the author noticed that there was perfect symmetry for both descent numbers and major indices. This led to the search for bijections that explain both symmetries.

\begin{definition}
A \textit{Young diagram} is a collection of boxes arranged in left-justified rows, with a weakly decreasing number of boxes in each row. A \textit{Young tableau} is a filling of a Young diagram that is
\begin{enumerate}
	\item weakly increasing across each row, and
	\item strictly increasing down each column.
\end{enumerate}

Every Young diagram corresponds to a partition $\lambda$, and we define the \textit{shape} of a Young tableau to be the partition that corresponds to the Young diagram.

Flipping a Young diagram $\lambda$ over its main diagonal (from upper left to lower right) gives the \textit{conjugate} diagram, denoted $\lambda'$.
\end{definition}

Given a tableau~$T$ and a positive integer $x$, an algorithm called \textit{row-insertion} constructs a new tableau, denoted $T \leftarrow x$, as follows. If $x$ is at least as large as all of the entries in the first row of $T$, $x$ is added in a new box to the end of the first row. If not, we find the left-most entry in the first row that is strictly larger than $x$, and we put $x$ in the box of this entry while removing or ``bumping'' the entry. We now take the bumped entry from the first row and repeat the process on the second row. The process is repeated until the bumped entry can be put at the end of the row it is bumped into, or until it is bumped out the bottom, in which case it forms a new row with one entry.

A row-insertion $T \leftarrow x$ determines a collection of boxes $R$, consisting of the boxes from which entries were bumped, together with the box where the last bumped element lands. This collection is called the \textit{bumping route} of the row-insertion, and the last box added to the diagram is called the \textit{new box} of the row-insertion. Given the new diagram and the new box, the row-insertion algorithm is reversible.

\begin{definition}
Given a partition $\lambda$, an \textit{oscillating tableau}~$T$ of length $n$ and shape $\lambda$ is a sequence of $n+1$ partitions
\begin{equation*}
	T = \left[ \emptyset = \lambda_0, \lambda_1, \dots, \lambda_n = \lambda \right]
\end{equation*}
such that $\lambda_i$ differs from $\lambda_{i-1}$ by a single box. Equivalently, we can think of $T$ as a walk on Young's lattice which starts at the empty partition, ends at $\lambda$, and has $n$ steps.
\end{definition}

Another operation one can do with Young tableaux is \textit{sliding}. A \textit{skew diagram} $\lambda / \mu$ is the diagram obtained by removing a smaller Young diagram $\mu$ from a larger Young diagram $\lambda$ containing $\mu$, and a \textit{skew tableau} is a filling of a skew diagram with positive integers, which are weakly increasing in rows and strictly increasing in columns. An \textit{inside corner} is a box in $\mu$ such that the boxes below and to the right are not in $\mu$, and an \textit{outside corner} is a box in $\lambda$ such that the boxes below and to the right are not in $\lambda$. The sliding operation, due to Sch{\"u}tzenberger, takes a skew tableau~$S$ and an inside corner, which can be thought of a hole, and slides the smaller of the box to the right and the box below into the hole. If both boxes are the same, the one below is chosen. This process creates a new hole and is repeated until the hole has been moved to an outside corner, i.e. there are no more holes.

In \cite{Sundaram}, Sundaram provides a bijection between matchings and oscillating tableaux of empty shape and length $2n$. Given a matching $\pi \in S_{2n}$, we define an oscillating tableau $T(\pi)$ of empty shape and length $2n$ as follows. We first define a sequence of Young tableaux $P_0, P_1, \dots, P_{2n}$ of length $2n$. Starting with $P_0 = \emptyset$, given $P_{i-1}$, we define $P_i$ as follows.
\begin{enumerate}
	\item If $i$ is the smaller element in its block $\left\{ i < j \right\}$ in $\pi$, let $P_i$ be obtained from $P_{i-1}$ by row-inserting $j$.
	\item If $i$ is the larger element in its block $\left\{ i > j \right\}$ in $\pi$, then $P_{i-1}$ contains $i$. $P_i$ is obtained by removing $i$ and sliding the hole out.
\end{enumerate}
If we define $T(\pi)$ to be the sequence of shapes
\begin{equation*}
	T(\pi) = \left( \text{sh}(P_0), \text{sh}(P_1), \dots, \text{sh}(P_{2n}) \right),
\end{equation*}
then $T$ is a bijective map from the set of matchings of $S_{2n}$ to the set of oscillating tableaux of empty shape and length $2n$.

For a given oscillating tableau $T = \left( \lambda_0, \dots, \lambda_{2n} \right)$, we can define the conjugate oscillating tableau to be $T' = \left( \lambda_0', \dots, \lambda_{2n}' \right)$. In \cite{Chen}, Chen et al. look at oscillating tableaux and their conjugates to prove certain identities.

We use the following result from \cite{Fulton} to prove the main result of this section.

\begin{lemma}[Row Bumping Lemma]\label{rowbump}
Consider two successive row-insertions, first row-inserting $x$ in a tableau $T$ and then row-inserting $x'$ in the resulting tableau $T \leftarrow x$, giving rise to two routes $R$ and $R'$, and two new boxes $B$ and $B'$.
\begin{enumerate}
	\item If $x \leq x'$, then $R$ is strictly left of $R'$, and $B$ is strictly left of and weakly below $B'$.
	\item If $x > x'$, then $R'$ is weakly left of $R$ and $B'$ is weakly left of and strictly below $B$.
\end{enumerate}
\end{lemma}

\begin{theorem}\label{thm3.1}
Let $\pi \in S_{2n}$ be a matching and define $\pi'$ by $\pi' = T^{-1}\left( T(\pi)' \right)$. Then,
\begin{equation*}
	d(\pi) + d(\pi') = 2(n + 1),
\end{equation*}
and
\begin{equation*}
	maj(\pi) + maj(\pi') = 2n^2.
\end{equation*}
\end{theorem}

Before we prove Theorem \ref{thm3.1}, we provide the following example in order to read the proof more easily.

\begin{example}
Consider the matching $\sigma = (1 \, 4)(2 \, 3)(5 \, 6) \in S_6$. Then, $\sigma$ has descents at positions 1, 2, 3, and 5, and so, $d(\sigma) = 5$ and $maj(\sigma) = 11$. The sequence of tableaux associated with $\sigma$ is
\begin{equation*}
	\left( \emptyset, \tableau{4}, \tableau{3 \\ 4}, \tableau{4}, \emptyset, \tableau{6}, \emptyset \right),
\end{equation*}
and so, the associated oscillating tableau is
\begin{equation*}
	T(\sigma) = \left( \emptyset, \tableau{\mbox{}}, \tableau{\mbox{} \\ \mbox{}}, \tableau{\mbox{}}, \emptyset, \tableau{\mbox{}}, \emptyset \right).
\end{equation*}
Then, the conjugate of $T(\sigma)$ is
\begin{equation*}
	T(\sigma)' = \left( \emptyset, \tableau{\mbox{}}, \tableau{\mbox{} & \mbox{}}, \tableau{\mbox{}}, \emptyset, \tableau{\mbox{}}, \emptyset \right),
\end{equation*}
and the matching associated with the conjugate oscillating tableau is $T^{-1}\left( T(\sigma)' \right) = (1 \, 3)(2 \, 4)(5 \, 6)$, which has descents at positions 2 and 5. Hence, $d\left( T^{-1}\left( T(\sigma)' \right) \right) = 3$ and $maj\left( T^{-1}\left( T(\sigma)' \right) \right) = 7$, and we see that $d(\sigma) + d\left( T^{-1}\left( T(\sigma)' \right) \right) = 2( 3 + 1 )$ and $maj(\sigma) + maj\left( T^{-1}\left( T(\sigma)' \right) \right) = 2 \cdot 3^2$.
\end{example}

\begin{proof}[Proof of Theorem \ref{thm3.1}]
Given a matching $\pi \in S_{2n}$, let
\begin{equation*}
	T(\pi) = \left( \lambda_0, \lambda_1, \dots, \lambda_{2n} \right).
\end{equation*}

For any $1 \leq i \leq 2n - 1$, consider the sizes of $\lambda_{i-1}$, $\lambda_i$, and $\lambda_{i+1}$. Then, by Lemma \ref{rowbump}, we see that
\begin{enumerate}
	\item if $\lvert \lambda_{i-1} \rvert < \lvert \lambda_i \rvert$ and $\lvert \lambda_i \rvert > \lvert \lambda_{i+1} \rvert$, then $\pi$ has a descent at position~$i$.
	\item if $\lvert \lambda_{i-1} \rvert > \lvert \lambda_i \rvert$ and $\lvert \lambda_i \rvert < \lvert \lambda_{i+1} \rvert$, then $\pi$ does not have a descent at position~$i$.
	\item if $\lvert \lambda_{i-1} \rvert < \lvert \lambda_i \rvert < \lvert \lambda_{i+1} \rvert$ and the box added the second time is in a strictly lower row than the box added the first time, then $\pi$ has a descent at position~$i$.
	\item if $\lvert \lambda_{i-1} \rvert < \lvert \lambda_i \rvert < \lvert \lambda_{i+1} \rvert$ and the box added the second time is in a weakly higher row than the box added the first time, then $\pi$ does not have a descent at position~$i$.
	\item if $\lvert \lambda_{i-1} \rvert > \lvert \lambda_i \rvert > \lvert \lambda_{i+1} \rvert$ and the box removed the first time is in a strictly lower row than the box removed the second time, then $\pi$ has a descent at position~$i$.
	\item if $\lvert \lambda_{i-1} \rvert > \lvert \lambda_i \rvert > \lvert \lambda_{i+1} \rvert$ and the box removed the first time is in a weakly higher row than the box removed the second time, then $\pi$ does not have a descent at position~$i$.
\end{enumerate}
For $i$ fixed, replacing $T(\pi)$ by $T(\pi)'$ has the effect of interchanging cases~3 and~4, as well as cases~5 and~6, while preserving cases~1 and~2. Since $T(\pi)$ has empty shape, case~1 has to happen for precisely one more value of~$i$ than case~2. Hence, letting $\mathbbm{1}^\pi_{i,j}$ denote the indicator function that position $i$ of $\pi$ falls into case $(j)$,
\begin{align*}
	d(\pi) + d(\pi') &= \sum_{i=1}^{2n-1} \mathbbm{1}_{i \in Des(\pi)} + \sum_{i=1}^{2n-1} \mathbbm{1}_{i \in Des(\pi')} + 2 \\
	&= \sum_{i=1}^{2n-1} \left( \mathbbm{1}^\pi_{i,1} + \mathbbm{1}^\pi_{i,3} + \mathbbm{1}^\pi_{i,5} \right) + \sum_{i=1}^{2n-1} \left( \mathbbm{1}^{\pi'}_{i,1} + \mathbbm{1}^{\pi'}_{i,3} + \mathbbm{1}^{\pi'}_{i,5} \right) + 2 \\
	&= \sum_{i=1}^{2n-1} \left( \mathbbm{1}^\pi_{i,1} + \mathbbm{1}^\pi_{i,1} + \mathbbm{1}^\pi_{i,3} + \mathbbm{1}^\pi_{i,4} + \mathbbm{1}^\pi_{i,5} + \mathbbm{1}^\pi_{i,6} \right) + 2 \\
	&= \sum_{i=1}^{2n-1} \left( \mathbbm{1}^\pi_{i,1} + \mathbbm{1}^\pi_{i,2} + \mathbbm{1}^\pi_{i,3} + \mathbbm{1}^\pi_{i,4} + \mathbbm{1}^\pi_{i,5} + \mathbbm{1}^\pi_{i,6} \right) + 3 \\
	&= 2n + 2.
\end{align*}
Similarly, noting that $\sum_{i=1}^{2n-1} i(\mathbbm{1}^\pi_{i,1} - \mathbbm{1}^\pi_{i,2}) = n$, we see that
\begin{align*}
	maj(\pi) + maj(\pi') &= \sum_{i=1}^{2n-1} \left( i\mathbbm{1}^\pi_{i,1} + i\mathbbm{1}^\pi_{i,3} + i\mathbbm{1}^\pi_{i,5} \right) + \sum_{i=1}^{2n-1} \left( i\mathbbm{1}^{\pi'}_{i,1} + i\mathbbm{1}^{\pi'}_{i,3} + i\mathbbm{1}^{\pi'}_{i,5} \right) \\
	&= \sum_{i=1}^{2n-1} \left( 2i \mathbbm{1}^\pi_{i,1} + i\mathbbm{1}^\pi_{i,3} + i\mathbbm{1}^\pi_{i,4} + i\mathbbm{1}^\pi_{i,5} + i\mathbbm{1}^\pi_{i,6} \right) \\
	&= \sum_{i=1}^{2n-1} i + n \\
	&= 2n^2.
\end{align*}
\end{proof}

\section{Normal Convergence}

Let $D_n$ be the number of descents of a permutation $\pi_n$ which is uniformly chosen from matchings. In \cite{Fulman1}, Fulman showed that
\begin{equation}
	\mathbb{E}\left[ t^{D_n} \right] = \frac{(1-t)^{2n+1}}{(2n-1)!!} \sum_{k=0}^\infty \binom{\frac{k(k+1)}{2}+n-1}{n} t^k,\label{4.1}
\end{equation}

which is a formal identity and an actual identity for $\lvert t \rvert < 1$, where $n!!$ denotes the odd factorial. Now, we define the normalized random variable

\begin{equation*}
	W_n = \frac{D_n - n}{\sqrt{n}}.
\end{equation*}

Recall that the moment generation function (MGF) of a random variable $X$ is defined as $M_X(s) = \mathbb{E}\left[ e^{sX} \right]$. In \cite{Curtiss}, Curtiss proved the following MGF analogue of the L\'{e}vy continuity theorem, which we will use to prove the normal convergence.

\begin{theorem}\label{thm4.1}
	Suppose we have a sequence $\left\{ X_n \right\}_{n=1}^\infty$ of random variables and there exists $s_0 > 0$ such that each MGF $M_n(s) = \mathbb{E}\left[ e^{s X_n} \right]$ converges for $s \in \left( -s_0, s_0 \right)$. If $M_n(s)$ converges pointwise to some function $M(s)$ for each $s \in \left( -s_0, s_0 \right)$, then $M$ is the MGF of some random variable $X$, and $X_n$ converges to $X$ in distribution.
\end{theorem}

In view of this result by Curtiss, it suffices to show that $M_{W_n}(s)$ converges pointwise to $M_W(s) = e^{s^2/12}$, which is the MGF of $\mathcal{N}\left( 0, \frac{1}{6} \right)$.

In the previous section, we showed that $D_n$ is symmetric about $n$. Hence, $n - D_n$ has the same distribution as $D_n - n$, and so, $M_{W_n}(s) = M_{W_n}(-s)$. Thus, we can assume $s > 0$ without loss of generality. Since $e^{-s/\sqrt{n}} < 1$, we can plug $t = e^{-s/\sqrt{n}}$ into (\ref{4.1}) to get
\begin{align*}
	M_{W_n}(s) &= \frac{e^{-s\sqrt{n}} \left( 1 - e^{-\frac{s}{\sqrt{n}}} \right)^{2n+1}}{(2n-1)!!} \sum_{k=0}^\infty \binom{\frac{k(k+1)}{2}+n-1}{n} e^{-\frac{ks}{\sqrt{n}}} \\
	&= e^{-s\sqrt{n}} \left( \frac{1 - e^{-\frac{s}{\sqrt{n}}}}{\frac{s}{\sqrt{n}}} \right)^{2n+1} \frac{\left( \frac{s}{\sqrt{n}} \right)^{2n+1}}{(2n)!} \sum_{k=0}^\infty \left( \prod_{j=0}^{n-1} \left( k^2 + k + 2j \right) \right) e^{-\frac{ks}{\sqrt{n}}}
\end{align*}

By considering the Taylor expansion $\frac{1 - e^{-x}}{x} = 1 - \frac{x}{2} + \frac{x^2}{6} + O(x^3)$, we see that
\begin{equation*}
	e^{-s\sqrt{n}} \left( \frac{1 - e^{-\frac{s}{\sqrt{n}}}}{\frac{s}{\sqrt{n}}} \right)^{2n+1} = e^{\frac{s^2}{12}} + O\left( n^{-1/2} \right),
\end{equation*}
as $n \rightarrow \infty$, where the implicit bound depends only on $s$. Now, we want to prove

\begin{lemma}\label{lem4.1}
	For each fixed $s > 0$,
	\begin{equation}
		\lim_{n \rightarrow \infty} \frac{\left( \frac{s}{\sqrt{n}} \right)^{2n+1}}{(2n)!} \sum_{k=0}^\infty \left( \prod_{j=0}^{n-1} (k^2 + k + 2j) \right) e^{-\frac{ks}{\sqrt{n}}} = 1.\label{4.2}
	\end{equation}
\end{lemma}

\begin{proof}
First, we start with the lower bound. For $k-1 \leq x \leq k$, we have
\begin{equation*}
	e^{-\frac{ks}{\sqrt{n}}} \geq \left( e^{-\frac{s}{\sqrt{n}}} \right)^{x+1},
\end{equation*}
and
\begin{equation*}
	\prod_{j=0}^{n-1} \left( k^2 + k + 2j \right) \geq x^{2n}.
\end{equation*}
Thus,
\begin{align*}
	\frac{\left( \frac{s}{\sqrt{n}} \right)^{2n+1}}{(2n)!} \sum_{k=0}^\infty \left( \prod_{j=0}^{n-1} (k^2 + k + 2j) \right) e^{-\frac{ks}{\sqrt{n}}} &\geq \frac{\left( \frac{s}{\sqrt{n}} \right)^{2n+1}}{(2n)!} e^{-\frac{s}{\sqrt{n}}} \int_0^\infty x^{2n} \left(e^{-\frac{s}{\sqrt{n}}}\right)^x \, dx \\
	&= \frac{ \left( \frac{s}{\sqrt{n}} \right)^{2n+1}}{(2n)!} e^{-\frac{s}{\sqrt{n}}} \frac{(2n)!}{\left( -\log e^{-\frac{s}{\sqrt{n}}} \right)^{2n+1}} \\
	&= e^{-\frac{s}{\sqrt{n}}}
\end{align*}
This provides a uniform lower bound for the left-hand side of (\ref{4.2}).

As for the upper bound, more work is required due to the competition between the size of $n$ and $k$. In order to control the sum, we divide the range of $k$ into 3 ranges and investigate the behavior of the corresponding truncated sums respectively.

\textbf{Small range $(k < \sqrt{n})$.} Note that
\begin{equation*}
	\prod_{j=0}^{n-1} \left( k^2 + k + 2j \right) = k^{2n} \prod_{j=0}^{n-1} \left( 1 + \frac{1}{k} + \frac{2j}{k^2} \right) \leq k^{2n} \left( 2^n \, n! \right).
\end{equation*}

Then, we have
\begin{align*}
	\sum_{k \leq \sqrt{n}} \left( \prod_{j=0}^{n-1} \left( k^2 + k + 2j \right) \right) e^{-\frac{ks}{\sqrt{n}}} &\leq 2^n \, n! \sum_{k \leq \sqrt{n}} k^{2n}e^{-\frac{ks}{\sqrt{n}}} \\
	&\leq 2^n \, n! \sqrt{n} \left( \sqrt{n} \right)^{2n} \\
	&\leq C n^{2n+1} 2^n e^{-n},
\end{align*}
where $C$ is independent of $n$. Thus,
\begin{equation*}
	\frac{\left( \frac{s}{\sqrt{n}}\right)^{2n+1}}{(2n)!} \sum_{k < \sqrt{n}} \left( \prod_{j=0}^{n-1} (k^2 + k + 2j) \right) e^{-\frac{ks}{\sqrt{n}}} \leq C \frac{s^{2n+1} e^n}{n^n 2^n} \rightarrow 0
\end{equation*}
as $n \rightarrow \infty$.

\textbf{Medium range $(\sqrt{n} < k < \varepsilon n)$ and large range $(k > \varepsilon n)$.} Next, we focus on the range $k > \sqrt{n}$. We begin by noting that
\begin{align*}
	\prod_{j=0}^{n-1} \left( k^2 + k + 2j \right) = k^{2n} \prod_{j=0}^{n-1} \left( 1 + \frac{1}{k} + \frac{2j}{k^2} \right) &\leq k^{2n} \exp\left\{ \sum_{j=0}^{n-1} \left( \frac{1}{k} + \frac{2j}{k^2} \right) \right\} \\
	&\leq k^{2n} e^{\frac{n}{k} + \frac{n^2}{k^2}}.
\end{align*}
Then, we get
\begin{align*}
	\frac{\left( \frac{s}{\sqrt{n}} \right)^{2n+1}}{(2n)!} \sum_{k>\sqrt{n}} \left( \prod_{j=0}^{n-1} (k^2 + k + 2j) \right) e^{-\frac{ks}{\sqrt{n}}} &\leq \frac{\left( \frac{s}{\sqrt{n}} \right)^{2n+1}}{(2n)!} \int_{\sqrt{n}}^\infty u^{2n} e^{\frac{n}{u} + \frac{n^2}{u^2}} e^{-\frac{su}{\sqrt{n}}} \, du \\
	&= \frac{1}{(2n)!} \int_s^\infty v^{2n} e^{-v} e^{\frac{s\sqrt{n}}{v} + \frac{s^2n}{v^2}} \, dv.
\end{align*}
The last integral is close to what we want, but we need control over the undesirable exponential term.

For the medium range $\sqrt{n} < k < \varepsilon n$, note that $v \mapsto v^{2n}e^{-v}$ is increasing on $(0,2n)$. Thus, there exists an absolute $C > 0$ such that
\begin{align*}
	\frac{1}{(2n)!} \int_s^{\varepsilon n} v^{2n} e^{-v} e^{\frac{s\sqrt{n}}{v} + \frac{s^2n}{v^2}} \, dv &\leq \frac{1}{(2n)!} (\varepsilon n)^{2n+1} e^{-\varepsilon n} e^{\sqrt{n} + n} \\
	&\leq C \sqrt{n} \varepsilon^{2n+1} e^{\sqrt{n} + (3 - \log 4 - \varepsilon)n},
\end{align*}
which decays exponentially for $\varepsilon > 0$ sufficiently small, i.e. $\varepsilon^2 \cdot \frac{e^3}{4} < 1$.

For the large range $k > \varepsilon n$,
\begin{align*}
	\frac{1}{(2n)!} \int_{\varepsilon n}^\infty v^{2n} e^{-v} e^{\frac{s\sqrt{n}}{v} + \frac{s^2n}{v^2}} \, dv &\leq \frac{1}{2n!} e^{\frac{s}{\varepsilon \sqrt{n}} + \frac{s^2}{\varepsilon^2 n}} \int_{\varepsilon n}^\infty v^{2n} e^{-v} \, dv \\
	&\leq 1 + O\left(n^{-\frac{1}{2}}\right).
\end{align*}

Combining the three calculations, we obtain the bound
\begin{equation*}
	\frac{\left( \frac{s}{\sqrt{n}} \right)^{2n+1}}{(2n)!} \sum_{k=0}^\infty \left( \prod_{j=0}^{n-1} \left( k^2+k+2j \right) \right) e^{-\frac{ks}{\sqrt{n}}} \leq 1 + o(1),
\end{equation*}
and so, the lemma is proven.
\end{proof}

Combining Theorem \ref{thm4.1} and Lemma \ref{lem4.1}, we obtain the desired result:

\begin{theorem}
Since $M_{W_n}(s)$ converges to $M_W(s)$ pointwise, $W_n$ converges to $W$ in distribution, i.e.
\begin{equation*}
	W_n \Rightarrow \mathcal{N}\left( 0, \frac{1}{6} \right).
\end{equation*}
\end{theorem}

Note that this is consistent with the result from the previous section where we found that the variance of $W_n$ is $\frac{n+3}{6}$.

\section{Acknowledgements}

The author would like to thank Brendon Rhoades for introducing the author to the world of Young tableaux and helping come up with the bijection in Section~3. The author would also like to thank Sangchul Lee for the suggestion of considering moment generating functions instead of characteristic functions and helping with the calculations of bounding the integrals in Section~4. Finally, the author would like to thank Jason Fulman for the suggestion of the problem and his guidance.

\bibliographystyle{amsplain}

\begin{thebibliography}{10}

\bibitem {Bayer} D. Bayer and P. Diacnois, \textit{Trailing the dovetail to its lair}, The Annals of Applied Probability, \textbf{2} (1992), no.2, 294--313.

\bibitem {Chatterjee} S. Chatterjee and P. Diaconis, \textit{A central limit theorem for a new statistic on permutations}, to appear in a special issue of Indian J. Pure App. Math.

\bibitem {Chen} W. Chen, E. Deng, R. Du, R. Stanley, and C. Yan, \textit{Crossings and nestings of matchings and partitions}, Transactions of the American Mathematical Society, \textbf{359} (2007), no.4, 1555--1575.

\bibitem {Curtiss} J. H. Curtiss, \textit{A note on the theory of moment generating functions}, The Annals of Mathematical Statistics, \textbf{13} (1942), no.4, 430--433.

\bibitem {Diaconis1} P. Diaconis, M. McGrath, and J. Pitman, \textit{Riffle shuffles, cycles, and descents}, Combinatorica, \textbf{15} (1995), no.1, 11--29.

\bibitem {Diaconis2} P. Diaconis and J. Pitman, \textit{Unpublished notes on descents}

\bibitem {Fulman1} J. Fulman, \textit{The distribution of descents in fixed conjugacy classes of the symmetric group}, Journal of Combinatorial Theory Series A, \textbf{84} (1998), no.2, 171--180.

\bibitem {Fulman2} J. Fulman, \textit{Stein's method and non-reversible Markov chains}, Stein's Method: Expository Lectures and Applications, IMS Lecture Notes Monogr. Ser., \textbf{46} (2004), 69--77.

\bibitem {Fulton} W. Fulton, \textit{Young tableaux}, Cambridge University Press, Cambridge, U.K. (1997)

\bibitem {Garsia} A. M. Garsia and I. Gessel, \textit{Permutation statistics and partitions}, Adv. Math, \textbf{31} (1979), no.3, 288--305.

\bibitem {Gessel} I. Gessel and C. Reutenauer, \textit{Counting permutations with given cycle structure and descent set}, Journal of Combinatorial Theory Series A, \textbf{64} (1993), no.2, 189--215.

\bibitem {Goldstein} L. Goldstein and Y. Rinott, \textit{A permutation test for matchings and its asymptotic distrbution}, Metron: International Journal of Statistics, \textbf{61} (2004), no.3, 375--388.

\bibitem {Harper} L. H. Harper, \textit{Stirling behavior is asymptotically normal}, Annals of Mathematical Statistics \textbf{38} (1966), 410--414.

\bibitem {Knuth} D. Knuth, \textit{The art of computer programming, Vol. 3. Sorting and searching}, Addison-Wesley, Reading, MA. (1973)

\bibitem {Petersen} T. K. Petersen, \textit{Eulerian numbers}, Birkh{\"a}user-Springer, New York, NY. (2015)

\bibitem {Sundaram} S. Sundaram, \textit{The Cauchy identity for $Sp(2n)$}, Journal of Combinatorial Theory Series A, \textbf{53} (1990), no.2, 209--238.

\end{thebibliography}

\end{document}